\documentclass[a4paper,12pt]{amsart}
\usepackage{amssymb,enumerate,a4wide}
\usepackage{rotating}

\newcommand{\FF}{{\mathbb{F}}}
\newcommand{\QQ}{{\mathbb{Q}}}
\newcommand{\ZZ}{{\mathbb{Z}}}

\newcommand{\bC}{{\mathbf C}}
\newcommand{\bG}{{\mathbf G}}
\newcommand{\bH}{{\mathbf H}}
\newcommand{\bL}{{\mathbf L}}
\newcommand{\bT}{{\mathbf T}}

\newcommand{\fS}{{\mathfrak{S}}}

\newcommand{\cE}{{\mathcal E}}
\newcommand{\cS}{{\mathcal S}}

\newcommand{\Ind}{\operatorname{Ind}}
\newcommand{\Inn}{\operatorname{Inn}}
\newcommand{\Irr}{\operatorname{Irr}}
\newcommand{\Aut}{\operatorname{Aut}}
\newcommand{\Out}{\operatorname{Out}}
\newcommand{\Sp}{\operatorname{Sp}}
\newcommand{\Spin}{\operatorname{Spin}}
\newcommand{\SO}{\operatorname{SO}}
\newcommand{\SC}{{\operatorname{sc}}}
\newcommand\RLG{{R_\bL^\bG}}
\newcommand\RtLtG{{R_\tbL^\tbG}}
\newcommand\RTG{{R_\bT^\bG}}
\newcommand\eq{{=}}
\newcommand{\Chevie}{{\sf Chevie}}

\let\eps=\epsilon
\let\de=\delta
\let\la=\lambda
\let\ti=\times
\let\tri=\bigtriangleup

\newcommand{\tG}{\tilde G}
\newcommand{\tbG}{{\widetilde\bG}}

\newcommand{\tL}{\tilde L}
\newcommand{\tbL}{{\widetilde\bL}}
\newcommand{\tchi}{\tilde\chi}
\newcommand{\tpsi}{\tilde\psi}
\newcommand{\trho}{\tilde\rho}
\newcommand{\ts}{\tilde s}
\newcommand{\tw}[1]{{}^{#1}\!}
\newcommand{\twd}{{}^\delta\!}
\newcommand{\twe}{{}^\eps\!}
\newcommand{\Ph}[1]{\Phi_{#1}}

\newtheorem{thm}{Theorem}[section]
\newtheorem{lem}[thm]{Lemma}
\newtheorem{cor}[thm]{Corollary}
\newtheorem{prop}[thm]{Proposition}

\newtheorem*{thmA}{Theorem 1}
\newtheorem*{thmB}{Theorem 2}

\theoremstyle{definition}
\newtheorem{exmp}[thm]{Example}

\theoremstyle{remark}
\newtheorem{rem}[thm]{Remark}

\begin{document}

\title{Cuspidal characters and automorphisms}

\date{\today}

\author{Gunter Malle}
\address{FB Mathematik, TU Kaiserslautern, Postfach 3049,
         67653 Kaisers\-lautern, Germany.}
\email{malle@mathematik.uni-kl.de}
\thanks{The author acknowledges financial support by ERC Advanced Grant 291512
 and by DFG TRR 195.}



\begin{abstract}
We investigate the action of outer automorphisms of finite groups of Lie type
on their irreducible characters. We obtain a definite result for cuspidal
characters. As an application we verify the inductive McKay condition for
some further infinite families of simple groups at certain primes.
\end{abstract}

\maketitle

%%%%%%%%%%%%%%%%%%%%%%%%%%%%%%%%%%%%%%%%%%%%%%%%%%%%%%%%%%%%%%%%%%%%%%%%%
\section{Introduction}

An important open problem in the ordinary representation theory of finite
groups of Lie type is to determine the action of outer automorphisms on
the set of their irreducible characters, and more generally to determine the
irreducible character degrees of the corresponding almost (quasi-)simple
groups. While the action of diagonal automorphisms and the corresponding
extension problems are well understood by the work of Lusztig \cite{Lu88},
based on the fact that such extensions can be studied in the framework of
finite reductive groups, much less is known in the case of field and also
of graph automorphisms.

The most elusive situation seems to be the one where irreducible characters
not stable under diagonal automorphisms are concerned. In \cite{MS16} we
obtained a certain reduction of this problem to the case of cuspidal characters.
This is the situation we solve here by applying methods and results from
block theory and Deligne--Lusztig theory (see Section~\ref{sec:main}):

\begin{thmA}   \label{thm:main}
 Let $G$ be a quasi-simple finite group of Lie type. For any cuspidal character
 $\rho$ of~$G$ there is a semisimple character $\chi$ in the rational Lusztig
 series of $\rho$ having the same stabiliser as $\rho$ in the automorphism group
 of $G$.
\end{thmA}

Observe that the action on semisimple characters is well-understood by the
theory of Gelfand--Graev characters, see \cite{Sp12}. For linear and unitary
groups this result was obtained by Cabanes and Sp\"ath \cite{CS15} and we use
it in our proof, for symplectic groups it follows from recent work of
Cabanes--Sp\"ath \cite{CS16} and of Taylor \cite{Tay16}; for the other types
it is new. For the proof we first consider quasi-isolated series, see
Sections~\ref{sec:class} and~\ref{sec:exc}. Here, we connect $\rho$ to $\chi$
either via a sequence of Brauer trees, in which case we also obtain information
on maximal extendibility (see Corollary~\ref{cor:ext-simple}), or of
Deligne--Lusztig characters.
\smallskip

As an application we verify the inductive McKay condition for some series of
simple groups of Lie type and suitable primes $\ell$ (see
Section~\ref{sec:ind McKay}):

\begin{thmB}   \label{thm:main2}
 Let $q$ be a prime power and $S$ a finite simple group $\tw2E_6(q)$,
 $E_7(q)$, $B_n(q)$ or $C_n(q)$. Let $\ell\equiv3\pmod4$ be a prime with
 $\ell|(q^2-1)$. Then $S$ satisfies the inductive McKay condition at $\ell$.
 In particular, the inductive McKay condition holds for $S$ at $\ell=3$.
\end{thmB}

This is a contribution to a programme to prove McKay's 1972 conjecture on
characters of $\ell'$-degree based on its reduction to properties of
quasi-simple groups, an approach which has recently led to the completion of
the proof in the case when $\ell=2$ (see \cite{MS16}).
\bigskip

\noindent{\bf Acknowledgement:} I thank David Craven for his
remark that the consideration of Brauer trees might be useful, and the
Institut Mittag-Leffler, where this remark was made, for its hospitality.
I also thank Frank L\"ubeck for help in compiling Table~\ref{tab:qi-class},
Britta Sp\"ath and Jay Taylor for their comments on an earlier version,
and the anonymous referee for helpful comments which led to a substantial
clarification of several proofs.

%%%%%%%%%%%%%%%%%%%%%%%%%%%%%%%%%%%%%%%%%%%%%%%%%%%%%%%%%%%%%%%%%%%%%%%%%
\section{Cuspidal characters in classical groups}   \label{sec:class}

Throughout the paper we fix the following notation. We let $\bG$ be a simple
simply connected linear algebraic group over an algebraically closed field of
characteristic~$p$ with a Frobenius map $F:\bG\rightarrow\bG$ inducing an
$\FF_q$-structure, and we set $G:=\bG^F$ the finite group of fixed points
under~$F$. It is well-known that $G$ then is
a finite quasi-simple group in all but finitely many cases (see
\cite[Thm.~24.17]{MT}), and furthermore all but finitely many quasi-simple
finite groups of Lie type can be obtained as $G/Z$ for some suitable central
subgroup $Z\le Z(G)$ (the exceptions being the Tits simple group and a few
exceptional covering groups, see e.g.~\cite[Tab.~24.3]{MT}). For us, a
``finite group of Lie type'' is any nearly simple group whose non-abelian
composition factor is neither sporadic nor alternating.
%% while a ``finite reductive group'' is the group of fixed points of a
%% reductive algebraic group under a Steinberg morphism.

Let $\bG\hookrightarrow\tbG$ be a regular embedding; thus $\tbG$ is a connected
reductive group with connected centre and derived subgroup equal to $\bG$.
For an extension $F:\tbG\rightarrow\tbG$ of the Frobenius map on $\bG$ we let
$\tG:=\tbG^F$. We choose a group $\tbG^*$ dual to
$\tbG$, with corresponding Frobenius map again denoted by $F$, and an
epimorphism $\pi:\tbG^*\rightarrow\bG^*$ dual to the regular embedding
$\bG\hookrightarrow\tbG$, and we write $\tG^*=\tbG^{*F}$ and $G^*=\bG^{*F}$
for the $F$-fixed points. Throughout, for closed $F$-stable subgroups $\bH$
of $\bG$, $\tbG$, $\bG^*,\ldots$ we will write $H:=\bH^F$ (in roman font) for
their group of fixed points.

Let us recall the description of automorphisms of a finite simple group of
Lie type: any automorphism of $S=G/Z(G)$ is a product of an inner, a diagonal,
a graph and a field automorphism. Here, the \emph{diagonal automorphisms} are
those induced by the embedding $G\hookrightarrow\tG$, the \emph{graph
automorphisms}
come from symmetries of the Dynkin diagram of $\bG$ commuting with $F$ and
\emph{field automorphisms} are induced by Frobenius maps on $\bG$ defining
a structure over some subfield $\FF_{q'}$ of $\FF_q$ some power of which is $F$
(see e.g.~\cite[Thm.~24.24]{MT}).

%%%%%%%%%%%%%%%%%%%%%%%%%%%%%%%%%%%
\subsection{Cuspidal unipotent characters on Brauer trees}
One crucial tool in our determination of the action of automorphisms is the
observation that cuspidal unipotent characters of classical groups lie in
blocks of cyclic defect for suitable primes. Lusztig gave a parametrisation
of the unipotent characters of groups $G$ of classical type in terms of
combinatorial objects called \emph{symbols}. According to this classification a
classical group has at most one cuspidal unipotent character, as recalled in
Table~\ref{tab:cusp-class},  which is thus in particular fixed by all
automorphisms of $G$. (The parameter $d_G$ occurring in the table will be used
in the statement of Lemmas~\ref{lem:def0} and~\ref{lem:in cyclic block}.)

\begin{table}[htbp]
\renewcommand{\arraystretch}{1.3}
$$\begin{array}{|l|ccccc|}
\hline
 G& A_{n-1}& \tw2A_{n-1}& B_n,C_n& D_n& \tw2D_n\\
\hline
 n& 1& a(a+1)/2& a(a+1)& a^2\text{ even}& a^2\text{ odd}\\
\text{label}& ()& (a,a-1,\ldots,1)& \binom{0\,\ldots\,2a}{-}& \binom{0\,\ldots\,2a-1}{-}& \binom{0\,\ldots\,2a-1}{-}\\
 d_G& -& 2(2a-1)& 4a& 2(2a-1)& 2(2a-1)\\
\hline
\end{array}$$
 \caption{Cuspidal unipotent characters in classical groups}   \label{tab:cusp-class}
\end{table}

The symbols parametrising unipotent characters behave very much like
partitions; in particular one can define hooks and cohooks, and the degrees of
the associated unipotent characters can be given in terms of a combinatorial
expression, called the hook formula (see e.g.~\cite{Ol}). Recall that for every
prime power $q$ and any integer $d>2$ with $(q,d)\ne(2,6)$ there exists a prime
dividing $q^d-1$ but no $q^f-1$ for $f<d$ called \emph{Zsigmondy (primitive)
prime of degree~$d$}. The following is easily checked from the hook formula:

\begin{lem}   \label{lem:def0}
 Let $\rho$ be a cuspidal unipotent character of a quasi-simple group $G$ of
 classical type. Then $\rho$ is of defect zero for every Zsigmondy prime of
 odd degree, as well as for those of even degree $d>d_G$.
\end{lem}

The blocks of cyclic defect and their Brauer trees for groups $G$ of classical
type have been determined by Fong and Srinivasan \cite{FS90}: assume that
$\ell\ne2$ is an odd prime and write $d=d_\ell(q)$ for the order of~$q$
modulo~$\ell$. First assume that $d$ is odd. Then a unipotent character
of $G$ lies in an $\ell$-block of cyclic defect if and only if the associated
symbol has at most one $d$-hook, and two unipotent characters lie in the
same $\ell$-block if their symbols have the same $d$-core. If $d=2d'$ is even,
the same statements hold with $d$ replaced by $d'$, ``hook'' replaced by
``cohook'' and ``core'' replaced by ``cocore''.  \par
Let us write $\Phi_d$ for the $d$th cyclotomic polynomial over~$\QQ$.

\begin{lem}   \label{lem:in cyclic block}
 Let $\rho$ be a cuspidal unipotent character of a quasi-simple group $G$ of
 classical type. Then there exist sequences of unipotent characters
 $\rho=\rho_1,\ldots,\rho_m=1_G$ of $G$ and of Zsigmondy primes $\ell_i\ne p$
 of either odd degree $d_i>2$ or even degree $d_i\ge d_G$ such that
 $\rho_i,\rho_{i+1}$ lie in the same $\ell_i$-block of cyclic defect of $G$,
 for $i=1,\ldots,m-1$, except when $G=D_4(2)$.
\end{lem}

\begin{proof}
The claim is clear for type $A_0$ as here the cuspidal character is the trivial
character. For $G$ of type $\tw2A_{n-1}$, $n\ge3$, there exists a cuspidal
unipotent character $\rho$ if and only if $n=a(a+1)/2$ for some $a\ge2$ (see
Table~\ref{tab:cusp-class}). This is labelled by the triangular partition
$\de_a=(a,\ldots,1)$ of $n$, which has a unique hook of length $2a-1$. By the
hook formula (see e.g.~\cite{Ol}) this implies that $|G|/\rho(1)$ is
divisible by $\Phi_d(q)$ exactly once, where $d=2(2a-1)$. So by \cite{FS90}
$\rho$ lies in an $\ell_1$-block of cyclic defect for any Zsigmondy prime
divisor $\ell_1$ of $q^{2a-1}+1$. (Such a prime $\ell_1$ exists unless
$(q,a)=(2,2)$, in which case $n=3$, but $\tw2A_2(2)$ is solvable.) Moreover,
the partition $(3a-3,a-3,\ldots,1)$ has the same $2a-1$-core as $\de_a$,
so the unipotent character $\rho_2$ labelled by it lies in the same
$\ell_1$-block as $\rho_1=\rho$. The latter partition has a
unique $4a-4$-hook, and arguing as before, we conclude that $\rho_2$ lies in
the same $\ell_2$-block of cyclic defect as the character $\rho_3$ labelled by
$(5a-10,a-5,\ldots,1)$, for $\ell_2$ a Zsigmondy prime divisor of
$q^{4a-4}+1$. Continuing inductively we arrive at $\rho_{2a+1}=1_G$ with
label the partition $(n)$.
\par
Groups of type $B_n$ and $C_n$, $n\ge2$, have a cuspidal unipotent character
$\rho$ if and only if $n=a(a+1)$ for some $a\ge1$. This is labelled by the
symbol $\binom{0\,\ldots\,2a}{-}$, which has a single $2a$-cohook.
So $\rho$ lies in a block of cyclic defect for any Zsigmondy prime divisor
$\ell_1$ of $q^{2a}+1$ (see again \cite{FS90}). Furthermore this block also
contains the unipotent character $\rho_2$ labelled by
$\binom{0\,\ldots\, 2a-2}{0\ 4a-1}$. This has a unique $4a-2$-hook, whose
removal gives $\binom{0\,\ldots\, 2a-2}{0\ 1}$, and adding a suitable
$4a-2$-hook we find $\binom{0\,\ldots\,2a-3\ 6a-4}{0\ 1}$, the symbol of a
unipotent character $\rho_3$ lying on the same Brauer tree as $\rho_2$ for
primes $\ell_3$ dividing $q^{4a-2}-1$. The Zsigmondy exception $a=2$
can be avoided by working with 5-hooks instead of 6-hooks in this step.
Continuing this way, after $m=a$ steps we arrive at the symbol $\binom{n}{-}$
labelling $\rho_{a+1}=1_G$.
\par
For groups of type $D_n$ or $\tw2D_n$, $n\ge4$, a cuspidal unipotent character
exists only if $n=a^2$ for some $a\ge2$. It is labelled by the symbol
$\binom{0\,\ldots\,2a-1}{-}$. This has a single $2a-1$-cohook and thus lies in
an $\ell_1$-block of cyclic defect for $\ell_1$ any primitive prime divisor of
$q^{2a-1}+1$. (Such a prime $\ell_1$ exists unless $(q,a)=(2,2)$, which leads
to the stated exception.) This block also contains the character $\rho_2$
labelled by $\binom{0\,\ldots\,2a-3}{0\ 4a-3}$. The latter symbol has a
unique $4a-4$-hook; removing this and adding a different one leads to the
symbol $\binom{0\,\ldots\,2a-4\ 6a-7}{0\ 1}$ of a unipotent character
$\rho_3$. Again, a straightforward induction completes the proof.
\end{proof}

%%%%%%%%%%%%%%%%%%%%%%%%%%%%%%%%%%%
\subsection{Constituents of Lusztig induction}
Recall from \cite[11.1]{DM91} that for any $F$-stable Levi subgroup $\bL$ of
a parabolic subgroup of $\bG$, Lusztig defines a linear map
$$\RLG:\ZZ\Irr(\bL^F)\longrightarrow\ZZ\Irr(\bG^F),$$
called \emph{Lusztig induction}. (In fact, this map might depend on the choice
of parabolic subgroup containing $\bL$, but it does not in the case of unipotent
characters, see e.g.~\cite[Thm.~1.33]{BMM}.) Also recall from \cite{BMM} that
an $F$-stable torus $\bT\le\bG$ is called a \emph{$d$-torus} (for some $d\ge1$)
if $\bT$ is split over $\FF_{q^d}$ but no non-trivial $F$-stable subtorus of
$\bT$ splits over any smaller field. The centralisers in $\bG$ of $d$-tori are
the \emph{$d$-split Levi subgroups}. They are $F$-stable Levi subgroups
of suitable parabolic subgroups of $\bG$.

\begin{lem}   \label{lem:mult 1}
 Let $G$ be quasi-simple of classical type $B_n,C_n,D_n$ or $\tw2D_n$ and
 $\rho$ be a cuspidal unipotent character of $G$. Let $\bL\le \bG$ be a
 2-split Levi subgroup of $\bG$
 \begin{itemize}
  \item of twisted type $\tw2A_{n-1}(q).\Ph2$ in types $B_n$ and $C_n$; or
  \item of twisted type $\tw2A_{n-2}(q).\Ph2^2$ in types $D_n$ and $\tw2D_n$.
 \end{itemize}
 There exist sequences of unipotent characters $\rho=\rho_1,\ldots,\rho_m=1_G$
 of $G$ and $\psi_1,\ldots,\psi_{m-1}$ of $L$ such that $\rho_i,\rho_{i+1}$
 both occur with multiplicity~$\pm1$ in $\RLG(\psi_i)$ for $i=1,\ldots,m-1$.
\end{lem}

Here, as in later results and tables, a notation like $\tw2A_{n-1}(q).\Ph2$ is
meant to indicate not the precise group theoretic structure of the finite
group, but rather the root system of the underlying algebraic group (not its
isogeny type) together with the action of the Frobenius: in our example the
underlying group has a root system of type $A_{n-1}$ on which $F$ acts by the
non-trivial graph automorphism, and its center is a 1-dimensional 2-torus.

\begin{proof}
First consider types $B_n$ and $C_n$. So by Table~\ref{tab:cusp-class},
$n=a(a+1)$ for some $a\ge1$, and $\rho$ is parametrised by the symbol
$\cS=\binom{0\,\ldots\,2a}{-}$. Observe that all unipotent characters of groups
of type $A$ are uniform, so that $\RLG(\rho)$ for any unipotent character
$\rho$ of $L$ can be expressed in terms of Deligne--Lusztig characters
$\RTG(1_T)$ for various $F$-stable maximal tori $\bT$ of~$\bG$. The
decomposition of Deligne--Lusztig characters has been determined explicitly
by Lusztig, and from this our claim can be checked by direct computation.
\par
To do this, we appeal to \cite[Thm.~3.2]{BMM}, which shows that Lusztig
induction of unipotent characters from 2-split Levi subgroups is, up to signs,
the same as induction in the corresponding relative Weyl groups. Since the
symbol $\cS$ has exactly $n$ 1-cohooks, its 1-cocore is trivial and so $\rho$
lies in the principal $2$-series. The relative Weyl group for the principal
2-Harish-Chandra series of $G$ is $W{=}W(B_n)$, the one for $L$ is its maximal
parabolic subgroup $W_L=\fS_n$. Determination of the 2-quotient of the symbol
$\cS$ shows that it corresponds to the character with label $(\de_a;\de_a)$
of $W$, with $\de_a=(a,\ldots,1)$ the triangular partition.
Let $\psi$ be the unipotent character of $L$ parametrised by the partition
$2\de_a=(2a,2a-2,\ldots,2)$. The constituents of $\Ind_{W_L}^W(2\de_a)$ are
exactly those bipartitions whose parts (including zeroes) can be added up so
as to obtain the partition $2\de_a$, with multiplicity one if this is possible
in a unique way. Thus it contains $(\de_a;\de_a)$ exactly once, but also
$(2a,2a-4,\ldots;2a-2,2a-6,\ldots)$. This in turn is contained once in
$\Ind_{W_L}^W((4a-2,4a-10,\ldots))$, and so on. Continuing in this way we
reach the symbol $(a(a+1);-)$ which parametrises the trivial character of~$G$.
\par
For $G$ of type $D_n$, by Table~\ref{tab:cusp-class} a cuspidal unipotent
character exists if $n=a^2$ for some even $a\ge2$, and it is labelled by the
symbol $\binom{0\,\ldots\,2a-1}{-}$. Again, this character lies in the principal
2-series of $G$, and by \cite[Thm.~3.2]{BMM} the decomposition of $\RLG$ can be
computed in the relative Weyl groups $W_L=\fS_{n-1}\le W=W(D_n)$. Here,
$\rho$ corresponds to the character of $W$ with label $(\de_a;\de_{a-1})$.
Let $W_1=\fS_n\le W(D_n)$ be a maximal parabolic subgroup of $W(D_n)$
containing $W_L$. For $\psi\in\Irr(W_L)$ labelled by
$\la=(\de_a+\de_{a-1})\setminus\{1\}$,
$\Ind_{W_L}^{W_1}(\la)$ contains all characters whose label is obtained by
adding one box to the Young diagram of $\la$. Then
$\Ind_{W_L}^W(\la)=\Ind_{W_1}^W\Ind_{W_L}^{W_1}(\la)$ can be
computed as before. It ensues that the multiplicities of the characters
labelled by $(\de_a;\de_{a-1})$ and by $(2a,2a-5,\ldots;2a-3,2a-7,\ldots)$ in
$\Ind_{W_L}^W(\la)$ are both~1. Again, an easy induction gives the claim.
The same type of reasoning applies for $\tw2D_n$ with $n$ an odd square.
\end{proof}

%%%%%%%%%%%%%%%%%%%%%%%%%%%%%%%%%%%
\subsection{Brauer trees and automorphisms}
The following result of Feit on Brauer trees will also provide some
information on extendibility. Let $\ell$ be a prime.

\begin{lem}   \label{lem:cycdef}
 Let $B$ be an $\ell$-block of a finite group $H$ with cyclic defect. Let
 $\gamma$ be an automorphism of $H$ fixing some non-exceptional character
 in $B$. Then $\gamma$ fixes every non-exceptional character in $B$.
\end{lem}

\begin{proof}
Assume that $\chi\in\Irr(B)$ is non-exceptional and fixed by $\gamma$.
Then by \cite[Thm.~2.4]{Ft} all nodes in the Brauer tree of $B$ are fixed by
$\gamma$, hence in particular all non-exceptional characters in $B$.
\end{proof}

\begin{cor}   \label{cor:extend}
 Let $N\unlhd H$ be finite groups with $H/N$ solvable and of order prime
 to~$\ell$. Let $\chi,\chi'$ be non-exceptional characters on the same
 $\ell$-Brauer tree for~$N$. Then $\chi$ extends to $H$ if and only if
 $\chi'$ does.
\end{cor}

\begin{proof}
As $H/N$ is solvable, there is a sequence of subgroups
$N=N_1\unlhd\cdots\unlhd N_r=H$ with $N_i/N_{i-1}$ cyclic of prime order.
Assume that $\chi$ extends to $\tilde\chi\in\Irr(H)$, and let
$\chi_i=\tilde\chi|_{N_i}$, $1\le i\le r$, a system of compatible extensions of
$\chi$ to $N_i$. \par
Assume that we have extended $\chi'$ to a character $\chi_i'$ of $N_i$ on the
same $\ell$-Brauer tree as $\chi_i$. Since $\chi_i$ extends to $H$, it is
invariant in $N_{i+1}$, hence by Lemma~\ref{lem:cycdef} the same is true for
$\chi_i'$. So $\chi_i'$ also extends to $N_{i+1}$ (as $N_{i+1}/N_i$ is cyclic),
and clearly we may choose an extension $\chi_{i+1}'$ in the
same $\ell$-block as $\chi_{i+1}$. So the claim follows by induction.
\end{proof}

\begin{rem}   \label{rem:lem2.4}
We can now lay out two types of arguments we will use to relate the stabilisers
in $\Aut(G)$ of two characters $\rho,\chi\in\Irr(G)$: \par
(a) Assume there is a sequence $\rho=\rho_1,\ldots,\rho_m=\chi$ of irreducible
characters of $G$, a sequence of primes $\ell_1,\ldots,\ell_{m-1}$ and a
sequence of $\ell_i$-blocks $B_1,\ldots,B_{m-1}$ of $G$ with cyclic defect such
that $\rho_i,\rho_{i+1}\in\Irr(B_i)$ are non-exceptional for all $i$. Then by
Lemma~\ref{lem:cycdef} any automorphism fixing $\rho=\rho_1$ also fixes
$\rho_m=\chi$ and vice versa.   \par
(b) Similarly, assume there is a Levi subgroup $L$ of $G$ and sequences
$\rho=\rho_1,\ldots,\rho_m=\chi$ of irreducible characters of $G$,
$\psi_1,\ldots,\psi_{m-1}$ of $L$ satisfying the conclusion of
Lemma~\ref{lem:mult 1}. Assume that $\gamma$ is an automorphism of $G$
stabilising $L$ and fixing all $\rho_i$ and all $\psi_i$, and that there
are $\gamma$-invariant normal subgroups $L'\le L$, $G'\le G$ such that all
$\rho_i,\psi_i$ have exactly two constituents upon restriction. Then if
$\gamma$ does not fix the constituents of $\rho_1$, it cannot fix those of
$\rho_m$ and vice versa.
\end{rem}

In our arguments we will have to deal not only with unipotent characters.
Recall that Lusztig gives a partition $\Irr(G)=\coprod_s\cE(G,s)$ of the set
of irreducible characters of $G$ into \emph{rational Lusztig series $\cE(G,s)$}
indexed by semisimple elements $s\in G^*$ up to conjugacy (see
\cite[Thm.~11.8]{B06}). Moreover, for any such $s$ the Lusztig series
$\cE(G,s)$ is in bijection with the unipotent characters of $C_{G^*}(s)$,
where, as customary, a character of $C_{G^*}(s)$ is called unipotent if its
restriction to $C_{\bG^*}^\circ(s)^F$ has unipotent constituents (see
\cite[Prop.~5.1]{Lu88}). We will be particularly interested in semisimple
characters. Recall our regular embedding $\bG\hookrightarrow\tbG$ with dual
epimorphism $\tbG^*\rightarrow\bG^*$ and let $\tilde s\in\tbG^{*F}$ be a
preimage of $s$. The \emph{semisimple character} $\tchi_s$ in
$\cE(\tG,\tilde s)$ can be defined as an explicit linear combination of
Deligne--Lusztig characters (see \cite[(15.6)]{B06}); the semisimple
characters in $\cE(G,s)$ are then just the constituents of the restriction of
$\tchi_s$ to $G$, see \cite[(15.8)]{B06}. 

We will need the following properties of Jordan decomposition:

\begin{lem}   \label{lem:Jordan}
 In the above situation we have:
 \begin{enumerate}
  \item[\rm(a)] Jordan decomposition sends semisimple characters to semisimple
   characters.
  \item[\rm(b)] $\cE(G,s)$ contains a cuspidal character if and only if
   $C_{G^*}(s)$ has a cuspidal unipotent character and moreover
   $Z^\circ(C_{\bG^*}(s))$ and $Z^\circ(\bG)$ have the same $\FF_q$-rank.
   In this case, Jordan decomposition induces a bijection between cuspidal
   characters.
 \end{enumerate}
\end{lem}

\begin{proof}
As explained above the semisimple character $\tchi_s\in\cE(\tG,\tilde s)$ is
uniform. Jordan decomposition preserves uniform functions, so $\tchi_s$ is sent
to the semisimple character in $\cE(C_{\tG^*}(\tilde s),1)$. As the
Deligne--Lusztig characters of $G$ are obtained by restriction from those of
$\tG$, the claim in~(a) follows from our definition of semisimple characters in
$\cE(G,s)$. \par
Part~(b) is pointed out for example in \cite[Rem.~2.2(1)]{KM15}.
\end{proof}

%%%%%%%%%%%%%%%%%%%%%%%%%%%%%%%%%%%
\subsection{Cuspidal characters in quasi-isolated series}
We now study the action of automorphisms on cuspidal characters in
quasi-isolated series of classical groups.

Let $s\in G^*$ be semisimple. Recall that $s$ is \emph{quasi-isolated in
$\bG^*$} if $C_{\bG^*}(s)$ is not contained in any proper $F$-stable Levi
subgroup of $\bG^*$. If $C_{\bG^*}^\circ(s)^F$ is a product of classical
groups, then it has a
unique cuspidal unipotent character (see Table~\ref{tab:cusp-class}), so the
cuspidal characters in $\cE(G,s)$ form a single orbit under diagonal
automorphisms. In particular, $\cE(G,s)$ can contain cuspidal characters not
fixed by some automorphism of $G$ stabilising $\cE(G,s)$ only if $C_{\bG^*}(s)$
is not connected.

\begin{thm}   \label{thm:qi-class}
 Let $G$ be quasi-simple of classical type $B,C,D$ or $\tw2D$, let $s\in G^*$
 be quasi-isolated and $\rho\in\cE(G,s)$ cuspidal. Then there is a semisimple
 character $\chi\in\cE(G,s)$ with the same stabiliser in $\Aut(G)$ as $\rho$.
\end{thm}

\begin{sidewaystable}[htbp]
 \caption{Disconnected centralisers of quasi-isolated elements in classical groups $\bG^*$}   \label{tab:qi-class}
$$\begin{array}{|l|cllc|l|l|l|}
\hline
     \bG^{*F}& o(s)& C_{\bG^*}^\circ(s)^F&& \kern-11pt A(s)^F& \text{conditions}& \text{$\cE_\text{cusp}(\bG^F\!,s)\ne\emptyset$ if}& \cr
\hline\hline
  B_n(q)& 2& B_{n-d}(q).\twd D_d(q)&& 2& 1\le d\le n& n-d=2\tri,\,d=\Box& \\
\hline
  C_n(q)& 2& C_{n/2}(q)^2&& 2& \text{$n$ even}& n=4\tri& \text{1)}\\
 (n\ge3)& 2& C_{n/2}(q^2)&& 2& \text{$n$ even}& n=4\tri& \\
    & 2& \twd A_{n-1}(q).(q-\delta1)&& 2& & \delta=-,\,n=\tri& \\
    & 4& C_d(q)^2.\twe A_{n-2d-1}(q).(q-\eps1)&& 2& 1\le d<\frac{n}{2}& d=2\tri,\,\eps\eq-,\,n-2d\eq\tri& \text{2)}\\
    & 4& C_d(q^2).\tw{-\eps} A_{n-2d-1}(q).(q+\eps1)&& 2& 1\le d<\frac{n}{2}& d=2\tri,\,\eps\eq+,\,n-2d\eq\tri& \\
\hline
  D_n(q)& 2& \twd D_d(q).\twd D_{n-d}(q)&& 2& 1\le d<\frac{n}{2}& d=\Box,\,n-d=\Box& \\
 (n\ge4)& 2& \twd D_{n/2}(q)^2&& 2^2& \text{$n$ even}& n=2\Box& \text{3)}\\
    & 2& D_{n/2}(q^2)& \kern-20pt(2\times)& 2^2& \text{$n$ even}& n=2\Box& \\
    & 2& \twd A_{n-1}(q).(q-\delta1)& \kern-20pt(2\times)& 2& \text{$n$ even}& \delta=-,\,n=\tri& \\
    & 4& \twd D_d(q)^2.\twe A_{n-2d-1}(q).(q-\eps1)&& 2^2& \text{$n$ even},\,1\le d<\frac{n}{2}& d=\Box,\,\eps\eq-,\,n-2d=\tri& \text{4)}\\
    & 4& D_d(q^2).\tw{-\eps}A_{n-2d-1}(q).(q+\eps1)& \kern-20pt(2\times)& 2^2& \text{$n$ even},\,1\le d<\frac{n}{2}& d=\Box,\,\eps\eq+,\,n-2d=\tri& \\
    & 4& \twd D_d(q)^2.A_{n-2d-1}(q).(q-1)&& 4& \text{$n$ odd},\,1\le d<\frac{n}{2},\,\eps\eq+& \text{never}& \\
    & 4& \tw2D_d(q^2).\tw2A_{n-2d-1}(q).(q+1)& \kern-20pt(2\times)& 4& \text{$n$ odd},\,1\le d<\frac{n}{2},\,\eps\eq+& d=\Box,\,n-2d=\tri& \\
    & 4& D_d(q).\tw2D_d(q).\tw2A_{n-2d-1}(q).(q+1)&& 2& \text{$n$ odd},\,1\le d<\frac{n}{2},\,\eps\eq-& d=1,\,n-2=\tri& \\
    & 4& D_d(q^2).A_{n-2d-1}(q).(q-1)&& 2& \text{$n$ odd},\,1\le d<\frac{n}{2},\,\eps\eq-& \text{never}& \\
\hline
\tw2D_n(q)& 2& \twd D_d(q).\tw{-\delta}D_{n-d}(q)&& 2& 1\le d< n/2& d=\Box,\,n-d=\Box& \\
 (n\ge4)& 2& D_{n/2}(q).\tw2D_{n/2}(q)&& 2& \text{$n$ even}& \text{never}& \\
    & 2& \tw2D_{n/2}(q^2)&& 2& \text{$n$ even}& n=2\Box& \\
    & 4& D_d(q).\tw2D_d(q).\twe A_{n-2d-1}(q).(q-\eps1)&& 2& \text{$n$ even},\,1\le d<\frac{n}{2}& d=1,\,\eps\eq-,\,n-2=\tri& \\
    & 4& \tw2D_d(q^2).\tw{-\eps}A_{n-2d-1}(q).(q+\eps1)&& 2& \text{$n$ even},\,1\le d<\frac{n}{2}& d=\Box,\,\eps\eq+,\,n-2d=\tri& \\
    & 4& D_d(q).\tw2D_d(q).A_{n-2d-1}(q).(q-1)&& 2& \text{$n$ odd},\,1\le d<\frac{n}{2},\,\eps\eq+& \text{never}& \\
    & 4& D_d(q^2).\tw2A_{n-2d-1}(q).(q+1)&& 2& \text{$n$ odd},\,1\le d<\frac{n}{2},\,\eps\eq+& d=\Box,\,n-2d=\tri& \\
    & 4& \twd D_d(q)^2.\tw2A_{n-2d-1}(q).(q+1)&& 4& \text{$n$ odd},\,1\le d<\frac{n}{2},\,\eps\eq-& d=\Box,\,n-2d=\tri& \text{5)}\\
    & 4& \tw2D_d(q^2).A_{n-2d-1}(q).(q-1)& \kern-20pt(2\times)& 4& \text{$n$ odd},\,1\le d<\frac{n}{2},\,\eps\eq-& \text{never}& \\
\hline
\end{array}$$
{\small Here $\eps\in\{\pm\}$ is such that $q\equiv\eps1\ (4)$,
  $\delta\in\{\pm\}$, $\tw\pm D_1(q)$ denotes a torus of order~$q\mp1$,
  $\Box$ a square, $\tri$ a triangular number.}
  \smallskip
\end{sidewaystable}

\begin{proof}
Let $\gamma\in\Aut(G)$. If $\gamma$ does not stabilise $\cE(G,s)$, then it
lies neither in the stabiliser of $\rho$ nor of any semisimple character
in $\cE(G,s)$. So we may assume that $\cE(G,s)$ is $\gamma$-stable.
Moreover, if $C_{\bG^*}(s)^F=C_{\bG^*}^\circ(s)^F$ then $\cE(G,s)$ contains a
unique cuspidal and a unique semisimple character, and again we are done. Thus,
as $|C_{\bG^*}(s):C_{\bG^*}^\circ(s)|$ divides $|Z(\bG)|$, which is a $2$-power
and prime to~$p$, we have in particular that $q$ is odd. We discuss the
remaining possibilities case-by-case.   \par
The classes of quasi-isolated elements in $\bG^*$ were classified by Bonnaf\'e
\cite[Tab.~2]{B05}. The various rational types are worked out in
Table~\ref{tab:qi-class}. Here $A(s):=C_{\bG^*}(s)/C_{\bG^*}^\circ(s)$ denotes
the group of components of the centraliser, $o(s)$ is the order of $s$ and the
structure of the abelian group $A(s)$ is indicated by giving the orders with
multiplicities of its cyclic factors. \par
Our strategy of proof is as follows. In each case,
$[C_{\bG^*}^\circ(s),C_{\bG^*}^\circ(s)]^F$ is a product of classical groups
$G_1\cdots G_r$. By Lemma~\ref{lem:Jordan}(b), the Lusztig series $\cE(G,s)$
contains a cuspidal character only if each of these factors $G_i$
has a cuspidal unipotent character $\rho_i$. If all factors are non-isomorphic,
then for the factor of largest rank, say $G_1$, take a Zsigmondy prime $\ell$
as in Lemma~\ref{lem:in cyclic block}. (Note that the exception $D_4(2)$ does
not occur here as $q$ is odd.) Then $\rho_1$ lies in an $\ell$-block of cyclic
defect, while all the other $\rho_i$ are of $\ell$-defect~0. As Jordan
decomposition preserves blocks with cyclic defect and their Brauer trees
(see \cite{FS90}), the Jordan correspondent $\rho$ of
$\rho_1\otimes\cdots\otimes\rho_r$ then also lies in a block of cyclic defect.
So we conclude with Remark~\ref{rem:lem2.4}(a) using the sequence of characters
from Lemma~\ref{lem:in cyclic block}.
\par
For example, in $G=\Sp_{2n}(q)$ centralisers of quasi-isolated elements $s$ in
$G^*=\SO_{2n+1}(q)$ with $|C_{\bG^*}(s):C_{\bG^*}^\circ(s)|=2$ have types
$\twd D_d(q)B_{n-d}(q)$ for $1\le d\le n$ and $\de\in\{\pm\}$. These possess
a cuspidal unipotent character $\rho_d\otimes\rho_{n-d}$ only if $d=a^2$
and $n-d=b(b+1)$ for some $a,b\ge1$. Now for $a>b$ the cuspidal unipotent
character $\rho_d$ of $\twd D_d(q)$ lies in a Brauer tree for primitive prime
divisors $\ell$ of $q^{2a-1}+1$, while $\rho_{n-d}$ is of $\ell$-defect zero
in $B_{n-d}(q)$; and for $a\le b$ we have that $\rho_{n-d}$ lies in a Brauer
tree for primitive prime divisors $\ell$ of $q^{2b}+1$, and $\rho_d$ is of
$\ell$-defect zero.
\par
Thus we are only left with those cases when $C_{G^*}(s)$ has two isomorphic
quasi-simple factors, and these are the factors of largest rank. The
corresponding lines are marked (1)--(5) in the last column of
Table~\ref{tab:qi-class}. We discuss them individually.
\par\smallskip\noindent
{\bf Cases~2, 4 and~5}: Note that here cuspidal characters only arise when
$\eps=-$, so when $q\equiv3\pmod4$. But then $q$ is not a square and so field
automorphisms have odd order. In Case~2 the outer automorphism group of $G$ is
the direct product of the diagonal automorphism group $A$ of order~2 with an
odd order group of field automorphisms. So the latter must act trivially on
all $A$-orbits in $\cE(G,s)$. The cuspidal characters as well as the semisimple
characters both lie in $A$-orbits of length~2, so we are done. \par
In Case~5 a Sylow 2-subgroup $A$ of the outer automorphism group of $G$
consists of the diagonal automorphism group $A_0$ of order~4 extended by the
cyclic group of graph-field automorphisms. This has as quotient a dihedral
group of order~8, as the graph-field automorphism of $\Spin_{2n}^-(q)$
centralises a subgroup $\Spin_{2n-1}(q)$ and so acts non-trivially on the
centre. \par
Clearly automorphisms of odd order must act trivially on any $A_0$-orbit in
$C_G(s)$ which they fix. Now let $\gamma\in A$ be an automorphism of 2-power
order moving a cuspidal character $\rho$. As there is just one class
of elements $s$ with the relevant centraliser, the image of $\rho$ must lie
in the $A_0$-orbit $R$ of $\rho$, and $A$ acts faithfully on $R$. The same then
holds for the $A$-orbit of semisimple characters in $\cE(G,s)$. In particular
$\gamma$ also moves some semisimple character and vice versa. \par
The same argument applies to Case~4: again a Sylow 2-subgroup of the outer
automorphism group has a dihedral quotient as the graph automorphism
interchanges the two half-spin groups.
\par\smallskip\noindent
{\bf Case~1}: Here, cuspidal characters occur if $n=2r=2a(a+1)$. Let $\bL\le\bG$
be an $F$-stable Levi subgroup of twisted type $\tw2A_{n-1}(q).\Ph2$ containing
a Sylow 2-torus of $C_{\bG^*}(s)$, with dual $\bL^*\le\bG^*$. Then
$C_{\bL^*}(s)$ is disconnected of type
$C_{\bL^*}(s)^F=(\tw2 A_{r-1}(q).\Ph2)\wr2$.
Let $\bG\hookrightarrow\tbG$ be our regular embedding, and $\tbL=\bL Z(\tbG)$.
Let $\ts$ be an $F$-stable preimage of $s$ in $\tbG^*$. Then
$C_{\tbG^*}(\ts)=\tbG_1^2$ is connected with $\tbG_1^F$ of type $C_r(q)$, and
$C_{\tbL^*}(\ts)=\tbL_1^2$ is connected with $\tbL_1^F$ of type
$\tw2A_{r-1}(q).\Ph2$.
By Lemma~\ref{lem:mult 1}(1) there is a chain of unipotent characters
$\psi_{1,i}$ of $\tL_1$ and of unipotent characters
$\rho_1=\rho_{1,1},\ldots,\rho_{1,m}=\chi_1$ of $\tG_1$ connecting the
cuspidal unipotent character $\rho_1$ to the semisimple character
$\chi_1=1_{G_1}$ such that $R_{\tbL_1}^{\tbG_1}(\psi_{1,i})$ contains
$\rho_{1,i},\rho_{1,i+1}$ exactly once. Jordan decomposition maps
$\cE(\tG,\ts)$ to $\cE(\tG_1,1)\times\cE(\tG_1,1)$, and it maps $\cE(\tL,\ts)$
to $\cE(\tL_1,1)\times\cE(\tL_1,1)$. As Lusztig induction of unipotent
characters commutes with products, this implies that
$R_{\tbL_1^2}^{\tbG_1^2}(\psi_{1,i}^{\otimes2})$ contains
$\rho_{1,i}^{\otimes2}$ and $\rho_{1,i+1}^{\otimes2}$ exactly once.
\par
Let $\trho_i\in\Irr(\tG)$ correspond to $\rho_{1,i}^{\otimes2}$ under Jordan
decomposition, and $\tpsi_i\in\Irr(\tL)$ correspond to $\psi_{1,i}^{\otimes2}$.
As $\tbL$ is of type $A$, all of its unipotent characters are uniform. Now
Jordan decomposition commutes with Deligne--Lusztig induction, so
$R_\tbL^\tbG(\tpsi_i)$ contains $\trho_i$ and $\trho_{i+1}$ exactly once.
By their description above, the restriction of $\tpsi_i$ to $L$ splits into two
constituents
$\psi_i,\psi_i'$, and the restriction of $\trho_i$ to $G$ splits into two
constituents $\rho_i,\rho_i'$. Let $\gamma$ be an automorphism of $G$, then we
may choose $L$ to be $\gamma$-stable, and all characters in the Lusztig series
$\cE(\tbG,\ts)$ and $\cE(\tbL,\ts)$ are $\gamma$-stable. We are thus in the
situation of Remark~\ref{rem:lem2.4}(b) and may conclude.
\par\smallskip\noindent
{\bf Case~3}: The argument here is very similar, using a Levi subgroup
$\bL\le\bG$ of type $D_2(q).\tw2A_{n-3}(q).\Ph2$ with $C_{\bL^*}(s)$
disconnected of type $\tw2A_{r-2}(q)^2.\Ph2^4.2^2$ where $n=2r=2a^2$, and
applying Lemma~\ref{lem:mult 1}(2). Observe that $n\ne4$ so that there are
no triality automorphisms and $L$ can again be chosen $\gamma$-invariant.
\end{proof}

\begin{rem}
By \cite{CS15} the conclusion of Theorem~\ref{thm:qi-class}
also holds for $\bG$ of type $A_n$. For type $C_n$ our result also follows
from the recent work of Cabanes--Sp\"ath \cite{CS16} and of Taylor \cite{Tay16}.
\end{rem}

%%%%%%%%%%%%%%%%%%%%%%%%%%%%%%%%%%%%%%%%%%%%%%%%%%%%%%%%%%%%%%%%%%%%%%%%%
\section{Cuspidal characters in exceptional groups}   \label{sec:exc}

We now turn to groups of exceptional type. Here, there exist Lusztig series
containing more than one cuspidal character even for groups with connected
centre, which makes the situation somewhat more involved.
\goodbreak

%%%%%%%%%%%%%%%%%%%%%%%%%%%%%%%%%%%
\subsection{Automorphisms in type $E_6$}
We start by considering $\bG$ simple simply connected of type $E_6$. Then $\bG$
has a graph automorphism of order~2 which we may choose to commute with~$F$.

The interesting case occurs when there exist non-trivial diagonal automorphisms
of $G$. If $F$ is untwisted, so $G=\bG^F=E_6(q)_\SC$, this happens when
$q=p^f\equiv1\pmod3$. Then $\Out(G)\cong\fS_3\times C_f$, with the first factor
inducing diagonal and graph automorphisms, the second the field automorphisms
if $p\equiv1\pmod3$, respectively the product of the field automorphism by the
graph automorphism if $p\equiv2\pmod3$. If $F$ is twisted, so
$G=\tw2E_6(q)_\SC$, we have non-trivial diagonal automorphisms for
$q=p^f\equiv-1\pmod3$. Then $q$ is not a square, so $f$ is odd, and again
$\Out(G)\cong\fS_3\times C_f$, with the symmetric group $\fS_3$ inducing the
diagonal automorphisms and the graph-field automorphism $\gamma$, and the
cyclic group $C_f$ inducing the field automorphisms. We need the following
elementary observation.

\begin{lem}   \label{lem:trivi}
 The group $\fS_3\times C_f$ has a unique action up to permutation equivalence
 on a set of three elements in such a way that the elements of order~3 in the
 first factor act non-trivially.
\end{lem}

\begin{proof}
If the elements of order~3 in the $\fS_3$-factor act non-trivially, then the
whole $\fS_3$-factor must act faithfully, via its natural permutation
representation. As the cyclic factor $C_f$ must centralise the $\fS_3$-factor
in this action, it can only act trivially. So the action is unique up to
permutation equivalence. 
\end{proof}

Thus we see that if $X\in\Irr(G)$ is an orbit (of length~3) under diagonal
automorphisms which is stable under the graph respectively graph-field
automorphism, then the action of $\Aut(G)_X$ on $X$ is uniquely determined.

%%%%%%%%%%%%%%%%%%%%%%%%%%%%%%%%%%%
\subsection{Cuspidal characters in quasi-isolated series}
Next, we consider cuspidal characters in quasi-isolated series in exceptional
type groups. Let $\bG$ be simple simply connected with dual $\bG^*$ such that
$G=\bG^F$ is of exceptional type, and let $\bG\hookrightarrow\tbG$ be a
regular embedding.

\begin{thm}   \label{thm:qi-exc}
 Let $G$ be quasi-simple of exceptional type, $s\in G^*$ quasi-isolated and
 $\rho\in\cE(G,s)$ cuspidal. Then there is a semisimple character
 $\chi\in\cE(G,s)$ with the same stabiliser as $\rho$ in $\Aut(G)$.
\end{thm}

\begin{proof}
The quasi-isolated elements $s\in G^*$ were classified by Bonnaf\'e \cite{B05}.
We deal with the various possibilities case-by-case. First consider elements
with connected centraliser $\bC=C_{\bG^*}(s)$. Here we claim that cuspidal
characters in $\cE(G,s)$ are fixed by all automorphisms. If $s=1$, so
$\rho\in\cE(G,1)$ is cuspidal unipotent then by results of Lusztig $\rho$ is
invariant under all automorphisms of $G$, see e.g.~\cite[Thm.~2.5]{Ma08},
as is the semisimple character $1_G\in\cE(G,1)$. Next assume that $s\ne1$.
If $\bC$ has only components of classical type and $\bC^F$ does not involve
$\tw3D_4$, then each such component has at most one cuspidal unipotent
character and one semisimple character in any Lusztig series, and hence by
Lemma~\ref{lem:Jordan}(a) the same is true for $\cE(G,s)$. The claim then
follows trivially. Connected centralisers of isolated elements $s\ne1$ with an
exceptional component are of type $E_6$ in $E_7$, or of types $E_6.A_2$ or
$E_7.A_1$ in type $E_8$. Now, the cuspidal unipotent characters of $E_6(q)$
lie on a Brauer tree for Zsigmondy primes dividing $\Phi_9$ together with the
trivial character. Thus they lie on such a Brauer tree together with the
semisimple character if they occur in quasi-isolated series of type $E_6$
or $E_6.A_2$. Our claim thus follows from Lemma~\ref{rem:lem2.4}(a).
Similarly two of the three cuspidal unipotent characters of $\tw2E_6(q)$ lie
on the $\Phi_{18}$-Brauer tree, and the third one is uniquely determined by its
degree; and the two cuspidal unipotent characters of $E_7(q)$ lie on a
$\Phi_{18}$-Brauer tree and we conclude as before.  \par
Thus we may now assume that $C_{\bG^*}(s)$ is not connected.  As
centralisers of semisimple elements are connected in a group whose dual has
connected centre, we are in one of three situations: $q\equiv1\pmod3$ for
$G=E_6(q)_\SC$, $q\equiv-1\pmod3$ for $G=\tw2E_6(q)_\SC$, or $q$ is odd for
$G=E_7(q)_\SC$. The various rational forms of the occurring types of
disconnected centralisers $C_{\bG^*}(s)$ with cuspidal unipotent characters
can be computed using \Chevie; they are collected in
Table~\ref{tab:qi-exc}, depending on certain congruence conditions on~$q$.
The column labelled ``$|\cE(G,s)|$" gives the number of regular orbits (of
length~3 or~2) under the group of diagonal automorphisms and of orbits of
length~1 respectively, and similar the last column gives the same information
for the subset of cuspidal characters. Note that we only need to concern
ourselves with the regular orbits under diagonal automorphisms, and that the
semisimple characters always form a regular orbit.

\begin{table}[htbp]
\[\begin{array}{|l|lr|c|l|l|}
\hline
 \bG^*& C_{\bG^*}(s)^F& & q& |\cE(G,s)|& |\cE_\text{cusp}(G,s)|\\
\hline\hline
%% E_6(q)&        D_4(q).\Ph1^2.3&       & \equiv1\ (3)& 8\ti3+2\ti1& 1\ti3\\
 E_6(q)&      \tw3D_4(q).\Ph3.3& (2\ti)& \equiv1\ (3)& 8\ti3& 2\ti3\\
\hline
\tw2E_6(q)& \tw2A_2(q)^3.3& & \equiv2\ (3)& 3\ti3+8\ti1& 1\ti3\\
 &         \tw2A_2(q^3).3& (2\ti)& \equiv2\ (3)& 3\ti3& 1\ti3\\
 &        D_4(q).\Ph2^2.3&       & \equiv2\ (3)& 8\ti3+2\ti1& 1\ti3\\
 &      \tw3D_4(q).\Ph6.3& (2\ti)& \equiv2\ (3)& 8\ti3& 2\ti3\\
\hline
%% E_7(q)&    E_6(q).\Ph1.2& & \equiv1\ (2)& 30\ti2& 2\ti2\\
 E_7(q)&      \tw2E_6(q).\Ph2.2& & \equiv1\ (2)& 30\ti2& 3\ti2\\
%% & D_4(q).A_1(q)^2.\Ph1.2& & \equiv1\ (4)& 20\ti2+18\ti1& 1\ti2\\
 & D_4(q).A_1(q)^2.\Ph2.2& & \equiv3\ (4)& 20\ti2+18\ti1& 1\ti2 \\
 &    \tw2A_2(q)^3.\Ph2.2& & \equiv5\ (6)& 9\ti2+9\ti1& 1\ti2\\
\hline
\end{array}\]
\caption{Cuspidal characters in Lusztig series of quasi-isolated elements
 with disconnected centralisers in exceptional types}   \label{tab:qi-exc}
\end{table}

We consider these in turn, starting with $G=E_6(q)_\SC$. It can be checked by
direct computation in \Chevie\ (see \cite{ChM}) that the two classes of
quasi-isolated elements are invariant under the graph automorphism. In
particular the corresponding Lusztig series must be invariant under the graph
automorphism of $G$. Moreover, the $\tG$-orbits of cuspidal characters in those
series are invariant by degree reasons. So our claim follows from
Lemma~\ref{lem:trivi}. The situation is entirely similar for the six
quasi-isolated series in $\tw2E_6(q)_\SC$.
\par
In $G=E_7(q)_\SC$, in the second case the cuspidal and the semisimple
characters are contained in Brauer trees for Zsigmondy primes for $\Phi_6$
(note that $q\ne2$ as $q$ is odd). By the Bonnaf\'e--Rouquier Morita
equivalence these characters are non-exceptional, so Remark~\ref{rem:lem2.4}
applies.
\par
In the first case, let $\bL^*\le\bG^*$ be a Levi subgroup of type
$A_2(q).A_1(q^3).\Ph3$ belonging to the nodes $1,2,3,5,7$ of the Dynkin
diagram in the standard Bourbaki numbering used for example in \Chevie.
A \Chevie-calculation shows that its dual Levi subgroup $\bL\le\bG$ has
disconnected centre, and $C_{\bL^*}(s)$ is of type $A_2(q).\Ph2\Ph3\Ph6.2$.
Application of \cite[Thm.~3.2]{BMM}
gives that $\RtLtG(\psi)$ contains the cuspidal character $\rho$ with label
$\tw2E_6[1]$ as well as the semisimple character in $\cE(\tG,\ts)$ exactly
once, where $\psi$ is the semisimple character in $\cE(\tL,\ts)$. So the
claim holds for $\rho$ by Remark~\ref{rem:lem2.4}(b). The two other cuspidal
characters in this series lie on a Brauer tree for Zsigmondy primes
dividing~$\Ph{18}$ and we can use Lemma~\ref{lem:cycdef}.
\par
Finally, for the last case, take a 6-split Levi subgroup $\bL^*\le\bG^*$ for
the nodes $2,5,7$, of type $A_1(q^3).\Ph6^2$, whose dual again has
disconnected centre, and with $C_{\bL^*}(s)$ of type $(q^3+1)\Ph6^2.2$. Then
$\RtLtG(\psi)$ contains the cuspidal character in $\cE(\tG,\ts)$ as well as
the semisimple character exactly once and we conclude as before.
\end{proof}

\begin{cor}   \label{cor:ext-simple}
 Let $\rho$ be a cuspidal character of a quasi-simple group $G$ of Lie type in
 a quasi-isolated series and assume we are neither in cases (3) or~(4) of
 Table~\ref{tab:qi-class} nor in the first case of Table~\ref{tab:qi-exc}. Then
 some $\tilde G$-conjugate of $\rho$ extends to its inertia group in the
 extension $\hat G$ of $G$ by graph and field automorphisms. In particular
 $\rho$ satisfies part~(ii) of the inductive McKay condition from
 \cite[Thm.~2.1]{MS16}. 
\end{cor}

\begin{proof}
For groups of types $A_n$ and $\tw2A_n$ this has been shown by Cabanes--Sp\"ath
\cite{CS15}. For the other types, by the proofs of Theorem~\ref{thm:qi-class}
and~\ref{thm:qi-exc} we have connected $\rho$ via a sequence of Brauer trees
or of Deligne--Lusztig characters to a semisimple character $\chi$ in the same
Lusztig series.
According to \cite[Prop.~3.4(c)]{Sp12} there exists a semisimple character
$\chi'$ in the $\tG$-orbit of $\chi$ that satisfies the cited condition~(ii)
and thus in particular extends to its inertia group $I$ in $\hat G$. Thus, the
corresponding $\tG$-conjugate $\rho'$ of $\rho$ also has inertia group~$I$.
\par
If $\rho,\chi$ (and hence $\rho',\chi'$) are connected via Brauer trees,
$\rho'$ extends to $I$ by Corollary~\ref{cor:extend}. In the other cases not
excluded in the statement, all Sylow subgroups of $\hat G/G$ are cyclic, and
so $\rho'$ also extends by elementary character theory. This yields part~(ii)
of the inductive McKay condition.
\end{proof}

%%%%%%%%%%%%%%%%%%%%%%%%%%%%%%%%%%%%%%%%%%%%%%%%%%%%%%%%%%%%%%%%%%%%%%%%%
\section{Proof of Theorem~1}   \label{sec:main}

We are now ready to prove our main result. Here, for finite groups $U\unlhd V$
with characters $\chi\in\Irr(V)$, $\psi\in\Irr(U)$ we write $\Irr(U\mid\chi)$
for the constituents of $\chi|_U$, and $\Irr(V\mid\psi)$ for those of $\psi^V$.

\begin{lem}   \label{lem:semisimple}
 Let $\bG$ be a connected reductive group with Frobenius map $F$ and
 $\bG_0\unlhd\bG$ a closed $F$-stable normal subgroup such that
 $[\bG,\bG]=[\bG_0,\bG_0]$. Let $\chi_0\in\Irr(\bG_0^F)$ and
 $\chi\in\Irr(\bG^F)$ be semisimple characters. Then:
 \begin{enumerate}
  \item[\rm(a)] All characters in $\Irr(\bG^F\mid\chi_0)$ and in
   $\Irr(\bG_0^F\mid\chi)$ are semisimple.
  \item[\rm(b)] Let $\pi:\bG^*\rightarrow\bG_0^*$ be the dual epimorphism and
   $s_0\in\bG_0^{*F}$ such that $\chi_0\in\cE(\bG_0^F,s_0)$. Then
   $|\cE(\bG^F,s)\cap\Irr(\bG^F\mid\chi_0)|=1$ for all $s\in\bG^{*F}$ with
   $\pi(s)=s_0$.
 \end{enumerate}
\end{lem}

\begin{proof}
Choose a regular embedding $\bG\hookrightarrow\tbG$, then
$\bG_0\hookrightarrow\bG\hookrightarrow\tbG$ is also regular. The first
statement now follow immediately from the definition of semisimple characters
as the Alvis--Curtis duals of regular characters, see \cite[\S15A]{B06}. For
part~(b) set
$G:=\bG^F$, $G_0:=\bG_0^F$, and $A_G(s):=(C_{\bG^*}(s)/C_{\bG^*}^\circ(s))^F$.
By Lusztig's result \cite[Thm.~11.12]{B06} the induction $\Ind_{G_0}^G(\chi_0)$
is multiplicity-free and by \cite[Prop.~15.13]{B06} it has exactly 
$|G:G_0|/|A_{G_0}(s_0):A_G(s)|$ constituents. On the other hand, by direct
counting this is
exactly the number of semisimple classes of $\bG^{*F}$ lying above the class of
$s_0$. By \cite[Prop.~11.7]{B06} and part~(a) any corresponding Lusztig series
will contain a character from $\Irr(\bG^F\mid\chi_0)$. The pigeonhole
principle now shows that each such series will contain exactly one such
character.
\end{proof}

\begin{proof}[Proof of Theorem~1]
Let $\bG$ be simple, simply connected and $F:\bG\rightarrow\bG$ such that
$G=\bG^F$ is quasi-simple. Let $\rho\in\Irr(G)$ be cuspidal and $s\in G^*$
such that $\rho\in\cE(G,s)$. Let $\bL^*\le\bG^*$ be a minimal $F$-stable Levi
subgroup of $\bG^*$ containing $C_{\bG^*}(s)$ with dual $\bL$. According to
\cite[Thm.~9.5]{Tay16} the Jordan decomposition between $\cE(G,s)$ and
$\cE(L,s)$ induced by $\RLG$ commutes with any $\gamma\in\Aut(G)_s$, and by
Lemma~\ref{lem:Jordan} it sends cuspidal characters to
cuspidal characters and semisimple characters to semisimple ones. Here
$\Aut(G)_s$ denotes the stabiliser in $\Aut(G)$ of $\cE(G,s)$. Note that by
construction $s$ is quasi-isolated in $\bL^*$. Thus we have reduced our
question to the corresponding one for quasi-isolated series in $L$.
\par
Let $\bL_0:=[\bL,\bL]$ and let $s_0\in\bL_0^*$ be the image of $s$ under the
natural epimorphism $\bL^*\rightarrow\bL_0^*$ induced by the embedding
$\bL_0\le\bL$. Clearly $s_0$ is quasi-isolated in~$\bL_0^*$. As $\bG$ is of
simply connected type, so is $\bL_0$ (see \cite[Prop.~12.14]{MT}), hence a
direct product $\bL_1\cdots \bL_r$ of $F$-orbits
$\bL_i$, $1\le i\le r$, of simple components of $\bL_0$. Correspondingly,
we have $\bL_0^F=L_1\cdots L_r$ with quasi-simple finite groups of Lie type
$L_i$. Any irreducible character $\chi$ of $L_0$ is an outer tensor product
of irreducible characters $\chi_i$ of the $L_i$, which are cuspidal,
respectively semisimple, if and only if $\chi$ is. Moreover, if
$\chi\in\cE(L_0,s_0)$ then $\chi_i\in\cE(L_i,s_i)$, with $s_i$ the image of
$s_0$ under the epimorphism $\bL_0^*\rightarrow\bL_i^*$ induced by the
embedding $\bL_i\hookrightarrow\bL_0$. Again, the $s_i$ are quasi-isolated in
$\bL_i^*$. Now for cuspidal characters in quasi-isolated series of quasi-simple
groups, our claim holds: For groups of types $A_n$ and $\tw2A_n$ it follows
from \cite{CS15}, for the other groups of classical type it is contained in
Theorem~\ref{thm:qi-class} and for groups of exceptional type all relevant
cases have been dealt with in Theorem~\ref{thm:qi-exc}.
\par
It remains to deduce the claim for $\bL$ from the one for $\bL_0$.
First, for $\chi\in\cE(L,s)$ and $\chi_0\in\Irr(L_0\mid\chi)$, we have by
Lemma~\ref{lem:semisimple}(a) that $\chi$ is semisimple if and only $\chi_0$  
is, and by the definition of cuspidality, $\chi$ is cuspidal if and only if
$\chi_0$ is. Furthermore, for $\chi$ semisimple we have
$\Irr(L\mid\chi_0)\cap\cE(L,s)=\{\chi\}$ by Lemma~\ref{lem:semisimple}(b).
Thus $\chi$ is uniquely determined by $\chi_0$ (given its Lusztig series).
Since our claim holds for $\bL_0$, any cuspidal character
$\chi_0\in\cE(L_0,s_0)$ has the same stabiliser in $L$ as the semisimple
characters in this series, and so again
$\Irr(L\mid\chi_0)\cap\cE(L,s)=\{\chi\}$ and $\chi$ is uniquely
determined by $\chi_0$ and $s$. Now note that $L_0$ is invariant under all
automorphisms of $L$. But then our claim for the Lusztig series $\cE(L,s)$
follows from the corresponding one for the Lusztig series $\cE(L_0,s_0)$ of
$L_0$.
\end{proof}

\begin{exmp}
 Let $F':\bG\rightarrow\bG$ be a Frobenius endomorphism commuting with $F$
 and such that $Z(G)\subset\bG^{F'}$. Let $s\in G^*$ such that $\cE(G,s)$ is
 stable under the field or graph-field automorphism $\sigma$ of $G$ induced
 by $F'$. Then any cuspidal character $\rho\in \cE(G,s)$ is $\sigma$-invariant.
 Indeed, by \cite[Thm.~3.5]{BH11} any such $\sigma$ fixes all regular
 characters in $\cE(G,s)$, hence by \cite[Thm.~9.5]{Tay16} the (Alvis--Curtis
 dual) semisimple characters, hence $\rho$ by Theorem~1.
\end{exmp}

%%%%%%%%%%%%%%%%%%%%%%%%%%%%%%%%%%%%%%%%%%%%%%%%%%%%%%%%%%%%%%%%%%%%%%%%%
\section{On the inductive McKay condition}   \label{sec:ind McKay}

We apply our previous result to verify the inductive McKay condition for
several series of simple groups at suitable primes. Let still $\bG$ be simple
of simply connected type, with regular embedding $\bG\hookrightarrow\tbG$
and dual epimorphism $\pi:\tbG^*\rightarrow\bG^*$.
By definition $\tG$ induces the full group of diagonal automorphisms of~$G$.
Let $D$ be the group of automorphisms of $\tG$ induced by inner, graph and
field automorphisms of $G$. We investigate the following property of characters
$\tilde\chi\in\Irr(\tG)$:
$$\textit{there exists $\chi\in\Irr(G\mid\tilde\chi)$ which is
  $D_{\Irr(G\mid\tilde\chi)}$-stable;}\eqno{(\dagger)}$$
here $D_{\Irr(G\mid\tilde\chi)}$ denotes the stabiliser in $D$ of the set
of irreducible characters of $G$ below~$\tchi$. The following is known:

\begin{lem}   \label{lem:semreg}
 Regular and semisimple characters of $\tG$ satisfy $(\dagger)$.
\end{lem}

\begin{proof}
First let $\tchi\in\Irr(\tG)$ be a semisimple character. Then the claim is
just \cite[Prop.~3.4(c)]{Sp12}. Regular characters are the images of semisimple
characters under the Alvis--Curtis duality (see \cite[14.39]{DM91}), whose
construction commutes with automorphism, so we conclude by the previous
consideration.
\end{proof}

In the next result, $\Irr_{\ell'}(G)$ denotes the set of irreducible characters
of $G$ of $\ell'$-degree.

\begin{prop}   \label{prop:dagger 2E6}
 Let $G=\tw2E_6(q)_\SC$ for a prime power $q$ and $\ell|(q+1)$ be a prime.
 Then any $\tchi\in\Irr(\tG\mid\Irr_{\ell'}(G))$ satisfies $(\dagger)$.
\end{prop}

\begin{table}[htbp]
\[\begin{array}{|lc|c|l|}
\hline
 C_{G^*}(s)& & q& |\cE(G,s)|\\
\hline\hline
    \tw2A_2(q^3).3& (2\ti)& \equiv2\ (3)& 3\ti3\\
   D_4(q).\Ph2^2.3&       & \equiv2\ (3)& 8\ti3+2\ti1\\
 A_1(q)^4.\Ph2^2.3&       & \equiv5\ (6)& 4\ti3+4\ti1\\
\hline
 A_1(q)^3.\Ph2^3.3&       & \equiv2\ (3)& 2\ti3+2\ti1\\
   A_1(q).\Ph2^5.3&       & \equiv2\ (3)& 2\ti3\\
          \Ph2^6.3&       & \equiv2\ (3)& 1\ti3\\
\hline
\end{array}\]
\caption{Some $\ell'$-series in $\tw2E_6(q)_\SC$, $\ell|(q+1)$}   \label{tab:2E6}
\end{table}

\begin{proof}
Let $\tilde s\in\tG^*$ be semisimple such that $\tchi\in\cE(\tG,\tilde s)$,
and let $s=\pi(\tilde s)$. If $\tchi$ restricts irreducibly to $G$, the claim
holds trivially. So we may assume $C_{\bG^*}(s)$ is not connected and hence
that $3|(q+1)$. From the list of character degrees of $\tw2E_6(q)_\SC$ provided
by L\"ubeck \cite{Lue}, we obtain the Table~\ref{tab:2E6} of conjugacy classes
of semisimple elements $s\in G^*$ with $C_{\bG^*}(s)^F\ne C_{\bG^*}^\circ(s)^F$
and such that $\cE(G,s)$ contains characters of $\ell'$-degree. As in
Table~\ref{tab:qi-exc} we also give the number of $\tG$-orbits
in $\cE(G,s)$ (this is implicit in the data in \cite{Lue}).   \par
In the first three entries of Table~\ref{tab:2E6} the element $s$ is
quasi-isolated, and the claim follows with Lemma~\ref{lem:trivi} as in the
proof of Theorem~\ref{thm:qi-exc}. In the other three cases the non-invariant
characters are either regular or semisimple, so they all satisfy ($\dagger$)
by Lemma~\ref{lem:semreg}.
\end{proof}

Thus we can prove our second main result:

\begin{proof}[Proof of Theorem~2]
If $\ell|(q-1)$ the claim is contained in \cite[Thm.~6.4(a)]{MS16}. So now
assume that $q\equiv-1\pmod\ell$.
For $S=\tw2E_6(q)$ all $\tchi\in\Irr(\tG\mid\Irr_{\ell'}(G))$ satisfy
property~$(\dagger)$ by Proposition~\ref{prop:dagger 2E6}. As moreover here
$D/\Inn(G)$ is cyclic, all such characters $\tchi$ satisfy condition~(ii)(1)
in  \cite[Thm.~2.1]{MS16} and thus the result follows from
\cite[Thm.~6.4(c)]{MS16}. For the other families of groups, we have that
$q\equiv-1\pmod\ell$ with $\ell\equiv3\pmod4$, so $q$ is not a square.
Thus $S$ has no even order field automorphisms, whence $\Out(G)$ is cyclic.
The claim then again follows from~\cite[Thm.~6.4(c)]{MS16}.
\end{proof}

It was shown in \cite[Cor.~7.3]{CS13} that the inductive McKay condition holds
for $\tw2E_6(q)$ at all primes $\ell\ge5$ if $q\not\equiv-1\pmod3$, as well
as for $E_7(q)$, $B_n(q)$ (and so $C_n(q)$) if $q$ is even.

%%%%%%%%%%%%%%%%%%%%%%%%%%%%%%%%%%%%%%%%%%%%%%%%%%%%%%%%%%%%%%%%%%%%%%%%%

\end{document}